\documentclass[a4paper, 12pt]{article}
\usepackage[utf8]{inputenc}
\usepackage{a4}
\usepackage[T1]{fontenc}
\usepackage[english]{babel}
\usepackage{amssymb,amsmath,amsthm}
\usepackage{enumerate}
\usepackage{graphicx}
\usepackage{epsfig}
\usepackage{comment}
\usepackage{xspace}
\usepackage{float}
\usepackage{multirow}
\usepackage{extarrows}
\usepackage{color}
\usepackage{nomencl}
\usepackage{nicefrac}
\makenomenclature
\newcommand{\N}{\ensuremath{\mathbb{N}}\xspace}

\newcommand{\R}{\ensuremath{\mathbb{R}}\xspace}

\newcommand{\eps}{\epsilon}

\renewcommand{\epsilon}{\varepsilon}

\newcommand{\mres}{\mathbin{\vrule height 1.6ex depth 0pt width 0.13ex\vrule height 0.13ex depth 0pt width 1.3ex}} 
\begin{document}

\numberwithin{equation}{section}
\newtheoremstyle{break}{15pt}{15pt}{\itshape}{}{\bfseries}{}{\newline}{}
\theoremstyle{break}
\newtheorem*{Satz*}{Theorem}
\newtheorem*{Rem*}{Remark}
\newtheorem*{Lem*}{Lemma}
\newtheorem{Satz}{Theorem}[section]
\newtheorem{Rem}[Satz]{Remark}
\newtheorem{Lem}[Satz]{Lemma}
\newtheorem{Prop}[Satz]{Proposition}
\newtheorem{Cor}[Satz]{Corollary}
\theoremstyle{definition}
\newtheorem{Def}[Satz]{Definition}

\parindent2ex

\begin{center}{\Large \bf A  density result for Sobolev functions and functions of higher order bounded variation with additional integrability constraints}
\end{center}

\begin{center}
 J. M\"uller
\end{center}


\begin{abstract}
We  prove density of smooth functions in subspaces of Sobolev- and higher order $BV$-spaces  of kind $W^{m,p}(\Omega)\cap L^q(\Omega-D)$ and $BV^m(\Omega)\cap L^q(\Omega-D)$,  respectively, where $\Omega\subset\R^n$ ($n\in\N$) is an open and bounded set with suitably smooth boundary, $m<n$ is a positive integer, $1\leq p<\infty$ s.t. $mp<n$, $D\Subset\Omega$ is a sufficiently regular open subset  and $q> np/(n-mp)$.  Here we say that a $W^{m-1,1}(\Omega)$-function is of $m$-th order bounded variation ($BV^m$) if its $m$-th  order partial derivatives in the sense of distributions are finite Radon measures. This takes up earlier results by C. Tietz and the author concerning functions with merely one order of differentiability which emerged in the context of a variational problem related to image analysis. In the connection of our methods we also investigate  a question concerning the boundary traces of $W^{1,p}(\Omega)\cap L^q(\Omega)$-functions.
\end{abstract}

\section{Introduction}
\label{intro}
In the study of a variational integral with applications to image processing, Christian Tietz and the author encountered the problem of approximating Sobolev and BV-functions which have additional summability properties  on a measurable subset of their domain. Namely we  considered  functionals of type
\begin{equation*}
 \mathcal{F}_{p,q}[u]=\intop_\Omega |\nabla u|^pdx+\intop_{\Omega-D}|u-f|^qdx 
\end{equation*} 
where $\Omega\subset\R^n$ is open and bounded with Lipschitz boundary, $D\subset\Omega$ is a measurable subset with $0<\mathcal {L}^n(D)<\mathcal{L}^n(\Omega)$, $f\in L^q(\Omega-D)$ is a given function and $u$ varies in  $W^{1,p}(\Omega)\cap L^q(\Omega-D)$, $1\leq p<q<\infty$. In case of $p=1$ one would rather study the problem $\mathcal{F}\rightarrow\min$ in the space $BV(\Omega)\cap L^q(\Omega-D)$ which naturally comes with a useful notion of compactness  in contrast to the non-reflexive space $W^{1,1}(\Omega)$ (see \cite{AFP}, Theorem 3.23, p. 132). 
For an outline of how the first and the second integral in the definition of $\mathcal{F}_{p,q}$ relate to the problems of image denoising and image inpainting, respectively, we would like to refer the interested reader to the introduction of \cite{FT}.\\
The following result revealed to be a key tool towards proving fine properties of solutions of $\mathcal{F}\rightarrow\min$:

\begin{Satz} [cf. \cite{FT}, Lemma 2.1 and Lemma 2.2]\label{Muti}
Let $\Omega\subset\R^n$ be open and bounded with Lipschitz boundary and $D\subset\Omega$ a measurable subset with $0<\mathcal {L}^n(D)<\mathcal{L}^n(\Omega)$.
\begin{enumerate}[(i)]
 \item If $u$ is in $W^{1,p}(\Omega)\cap L^q(\Omega-D)$, then there is a sequence of smooth functions $(\varphi_k)_{k=1}^\infty\subset C^\infty(\overline{\Omega})$ such that
\[
\|u-\varphi_k\|_{1,p;\Omega}+\|u-\varphi_k\|_{q;\Omega-D}\rightarrow 0\quad\text{ for }\quad k\rightarrow\infty.
\]
\item If $u$ is in $BV(\Omega)\cap L^q(\Omega-D)$, then there is a sequence of smooth functions $(\varphi_k)_{k=1}^\infty\subset C^\infty(\overline{\Omega})$ such that
\begin{align*}
 \|u-&\varphi_k\|_{1;\Omega}+\|u-\varphi_k\|_{q;\Omega-D}+\Bigl||\nabla u|(\Omega)-\intop_{\Omega}|\nabla\varphi_k|dx\Bigr|\\
&+\Bigl|\sqrt{1+|\nabla u|^2}(\Omega)-\intop_{\Omega}\sqrt{1+|\nabla \varphi_k|^2}dx\Bigr|\rightarrow 0\quad\text{ for }\quad k\rightarrow\infty.
\end{align*}

\end{enumerate}
\end{Satz}
Here, for a finite Radon measure $\mu$ the notation $|\mu|(\Omega)$ means the total variation and the expression $\sqrt{1+|\mu|^2}(\Omega)$ is defined in  the sense of convex functions of a measure as described in \cite{DT}: let $\mu=\mu^a(\mathcal{L}^n\mres\Omega)+\mu^s$ be the decomposition of $\mu$ into an absolutely continuous part w.r.t. the restriction of the $n$-dimensional Lebesgue measure to $\Omega$ with density $\mu^a\in L^1(\Omega)$  and a singular part $\mu^s\perp (\mathcal{L}^n\mres\Omega)$. Then, we define a measure $\sqrt{1+|\mu|^2}$ by setting
\[
 \sqrt{1+|\mu|^2}(B):=\big(\sqrt{1+|\mu^a|^2}(\mathcal{L}^n\mres{\Omega})\big)(B)+|\mu^s|(B)=\intop_{B}\sqrt{1+|\mu^a|^2}dx+|\mu^s|(B)
\]
for any Borel-set $B\subset\Omega$.
The aim of this note is to generalize  Theorem \ref{Muti} towards spaces of functions with higher order derivatives. The main results are:
\newpage
\begin{Satz}\label{Res1}
Let $\Omega\subset\R^n$ be open and bounded with Lipschitz boundary, $D\Subset\Omega$  an open precompact subset  with minimally smooth boundary\footnote{The term 'minimally smooth boundary' was coined by E.M. Stein in his book \cite{St}, p. 189 and refines the notion of a Lipschitz boundary slightly (for an explanation cf. section \ref{sec:2}).} and $u\in W^{m,p}(\Omega)\cap L^q(\Omega-D)$. Then there is a sequence of smooth functions $(\varphi_k)_{k=1}^\infty\subset C^\infty(\overline{\Omega})$ such that
\[
 \|u-\varphi_k\|_{m,p;\Omega}+\|u-\varphi_k\|_{q;\Omega-D}\rightarrow 0\quad\text{ for }\quad k\rightarrow\infty.
\]
\end{Satz}

\begin{Satz}\label{Res2}
Let $\Omega\subset\mathbb{R}^n$ be open and bounded with $C^1$-boundary, $D\Subset\Omega$ an open precompact subset with $C^1$-boundary  which is star-shaped with respect to a point  $x_0\in D$  and $u\in BV^m(\Omega)\cap L^q(\Omega-D)$. Then there is a sequence of smooth functions $(\varphi_k)_{k=1}^\infty\subset C^\infty(\overline{\Omega})$ such that
\begin{align*}
 \|u-&\varphi_k\|_{m-1,1;\Omega}+\|u-\varphi_k\|_{q;\Omega-D}+\Bigl||\nabla^mu|(\Omega)-\intop_{\Omega}|\nabla^m\varphi_k|dx\Bigr|\\
&+\Bigl|\sqrt{1+|\nabla^mu|^2}(\Omega)-\intop_{\Omega}\sqrt{1+|\nabla^m\varphi_k|^2}dx\Bigr|\rightarrow 0\quad\text{ for }\quad k\rightarrow\infty.
\end{align*}
\end{Satz}

The interest in a corresponding version of Theorem \ref{Muti} for higher orders of differentiability originates in the consideration of the functional which one gets after replacing the  gradient operator in the definition of $\mathcal{F}_{p,q}$ by its higher order analogue, $\nabla^mu:=\big(\tfrac{\partial}{\partial x_{i_1}}...\tfrac{\partial}{\partial x_{i_m}}u\big)_{i_1,...,i_m=1}^n$ which yields the functional
\[
\mathcal{F}_{m,p,q}[u]=\intop_{\Omega}|\nabla^mu|^pdx+\intop_{\Omega-D}|u-f|^qdx
\]
for $1<p<\infty$ and 
\[
 \mathcal{F}_{m,1,q}[u]=|\nabla^m u|(\Omega)+\intop_{\Omega-D}|u-f|^qdx.
\]
for $p=1$. For special choices of $\Omega$ and $f$, solutions of $\mathcal{F}\rightarrow\min$ can be interpreted in the context of higher order denoising/inpainting of images, which is a current field of investigation in image analysis, see, e.g., \cite{BKP}. As for $m=1$, an adequate approximation result  in the spirit of Theorem \ref{Muti} is useful for the investigation of (generalized) minimizers of $\mathcal{F}_{m,p,q}$. In this note, however, we restrict ourselves to the proofs of Theorems \ref{Res1} and \ref{Res2}  and postpone their applications to variational problems of higher order to a separate paper.

One should note at this point, that due to  Sobolev's embedding theorem we have for $mp<n$, that any function $u\in W^{m,p}(\Omega)$ is at least $np/(n-mp)$-summable and as a direct consequence of this and the embedding $BV(\Omega)\hookrightarrow L^{\nicefrac{n}{(n-1)}}(\Omega)$, any $u\in BV^m(\Omega)$ is $n/(n-m)$-summable; so an actual problem does not arise unless $q$ is 'large enough', which we want to assume tacitly from now on.

The methods for proving Theorem \ref{Muti}  were customized to grasp the case of merely one order of differentiability and fail for the general case since they crucially rely on a 'cut-off' procedure which turns out to be unsuitable  owing to the appearance of higher order terms from the iterated chain rule. So we had to pursue an entirely different approach which involves extending functions from $\Omega$ to $\R^n$ as well as a 'blow-up'-type argument, and therefore, unfortunately, goes along with much more rigorous restrictions on the geometry of $\Omega$ and $D$. Considerations on how to weaken the assumptions of Theorem \ref{Res1} led to a result on the boundary traces of Sobolev functions in the space $W^{1,p}(\Omega)\cap L^q(\Omega)$ which, albeit it does not confirm our expectations,  may be interesting in its own right for this very reason:
\begin{Satz}\label{Res3}
Let $\Omega=\R^{n-1}\times (0,\infty)\subset\R^n$, $1\leq p<\infty$, $\frac{np}{n-p}<q<\infty$ and $T:W^{1,p}(\Omega)\rightarrow L^p(\partial\Omega)=L^p(\R^{n-1})$ denote the trace map. Then, the following holds:
\begin{enumerate}[(i)]
 \item The images  $T\left(W^{1,1}(\Omega)\cap L^q(\Omega)\right)$ and $T\left(W^{1,1}(\Omega)\right)=L^1(\R^{n-1})$  coincide for any $1\leq q<\infty$.
 \item For $1<p<\infty$, the image  $T\left(W^{1,p}(\Omega)\cap L^q(\Omega)\right)$  is a proper subspace of $T(W^{1,p}(\Omega))=W^{1-\nicefrac{1}{p},p}(\R^{n-1})$.
\end{enumerate}
 
\end{Satz}

\begin{Rem}
 Although Theorem \ref{Res3} is formulated for the special case of $\Omega$ being a half-space, it extends to arbitrary Lipschitz domains via the standard procedure of localizing  with a suitable partition of unity and then retracting the general case to the half-space setting by piecewise  flattening the boundary.
\end{Rem}

 At this point I want to express particular thanks to Prof. Dr. M. Bildhauer of Saarland University for many fruitful discussions as well as to Prof. Dr. M. Fuchs, my PhD advisor, for directing my interest upon this topic. Further thanks go to Christian Tietz for valuable feedback and assessment. Finally I would like to thank Prof. Dr. J. Weickert for  supporting  my research  both financially and with his advice whenever it comes to questions from the field of image analysis.

The subsequent section introduces most of our (non-standard) notation and in particular explains  our conception of higher order bounded variation. It is followed by a section which gathers some needful results on Sobolev functions which might be common but can hardly be found in literature. Finally, sections \ref{sec:4} and \ref{sec:5} treat the proofs of Theorems \ref{Res1} and \ref{Res2}, respectively. The last section is devoted to the proof of Theorem \ref{Res3}.

\section{Preliminaries}
\label{sec:1}

\subsection{Notation and conventions, the space $BV^m(\Omega)$}
\label{sec:2}

Throughout the following, unless otherwise mentioned, $\Omega$ denotes an at least open and bounded subset of Euclidean space $\big(\R^n,|\,\cdot\,|\big)$  for $n\in\N$  with Lipschitz-regular boundary and $D\Subset\Omega$ is  an open, precompact subset with Lipschitz-boundary as well.  We adopt the notion of 'minimally smooth' boundaries from \cite{St} which means that there is an $\eps>0$, a  covering $(U_i)_{i=1}^\infty$ of $\partial\Omega$ through open sets, an integer $N$ and a positive real $L$ such that the following three conditions hold:
\begin{enumerate}[(i)]
 \item If $x\in\partial \Omega$, then $B_\eps(x)\subset U_i$ for some $i$.
 \item No point of $\R^n$ is contained in more than $N$ of the $U_i$'s.
 \item For each $i$, there are coordinates $(x_1,...,x_n)$ s.t. $\Omega\cap U_i$ can be written as $\{(x_1,...,x_n)\in U_i\;:x_n<\varphi_i(x_1,...,x_{n-1})\;\}$ with a Lipschitz-continuous function $\varphi_i:\R^{n-1}\rightarrow\R$ and $\text{Lip}(\varphi_i)\leq L$.
\end{enumerate}
The class of all sets with minimally smooth boundary  contains, e.g., open and bounded convex sets or open and bounded sets with $C^1$-boundary.
With $\Omega^\eps$ ($\Omega_\eps$) we denote the outer (inner) parallel set of $\Omega$ in distance $\eps$:
\[
\Omega^\eps:=\{x\in \R^n\;:\;\text{dist}(x,\Omega)<\eps\},\quad \Omega_\eps:=\{x\in\Omega\;:\;\text{dist}(x,\partial\Omega)>\eps\}.
\]
By $\rho_\eps\ast u$ we abbreviate the convolution of a function $u\in L^1_{\text{loc}}(\Omega)$ with a symmetric mollifier $\rho_\eps\in C_0^\infty(\R^n)$, which is supported in the closure of the ball $B_\eps(0)$. $\mathcal{H}^s$, $s>0$ designates the $s$-dimensional Hausdorff measure and
by $L^p(\Omega)$, $1\leq p<\infty$ we mean the space of (real-valued) functions which are $p$-integrable w.r.t. the $n$-dimensional Lebesgue measure $\mathcal{L}^n$, normed in the usual way by $\|\,\cdot\,\|_{p;\Omega}$. Further, $W^{m,p}(\Omega)$, $m\in\N$, designates the Sobolev space of (real-valued) functions whose distributional derivatives up to order $m$ are represented by $p$-integrable functions, endowed with the norm
\[ 
\|u\|_{m,p;\Omega}:=\sum_{\substack{\nu\in \N_0^n\\|\nu|\leq m}}\|\partial^\nu u\|_{p;\Omega}.
\]
The notion $\nabla^ku$ means the $k$-th iterated (distributional) gradient of a function $u$, i.e. the $k$-th order symmetric tensor-valued function with components $(\nabla^ku)_{i_1,...,i_k}=\partial_{i_1}\cdots\partial_{i_k} u$, $i_1,...,i_k\in\{1,...,n\}$. $S^k(\R^n)$ denotes the set of all symmetric tensors of order $k$ with real components, which is naturally isomorphic to the set of all $k$-linear symmetric maps $(\R^n)^k\rightarrow\R$.\\
 We  declare by 
\[
BV^m(\Omega):=\big\{u\in W^{m-1,1}(\Omega)\,:\,\nabla^{m-1}u\in BV(\Omega, S^{m-1}(\R))\big\}
\]
 the space of (real valued) functions of $m$-th order bounded variation, i.e. the set of all functions, whose distributional gradients up to order $m-1$ are  represented through 1-integrable tensor-valued functions and whose $m$-th distributional gradient is  a tensor-valued Radon measure of finite total variation 
\[
 |\nabla^mu|(\Omega)=\sup\Bigg\{\intop_\Omega u\Big(\sum_{{\nu\in\N_0^n,\,|\nu|=m}}\partial^\nu g_\nu\Big)dx\;:\;g\in C^m_0(\Omega,\R^M),\,\|g\|_\infty\leq 1\Bigg\}
\]
with $M:=\#\{\nu\in\N_0^n\;:\;|\nu|:=\nu_1+...+\nu_n=m\}$. Together with the norm
\[ 
\|u\|_{BV^m(\Omega)}:=\|u\|_{m-1,1;\Omega}+|\nabla^mu|(\Omega), 
\]
$BV^m(\Omega)$ becomes a Banach space. 

Spaces of this kind have been studied (in an even more general setting) in \cite{DT} and just like there, we will provide  $BV^m(\Omega)$ with another topology apart from the norm topology, induced by the following distance:\\
For $u,v\in BV^m(\Omega)$ we set
\begin{align*}
d_f(u&,v):=\\
&\|u-v\|_{m-1,1;\Omega}+\Bigl||\nabla^mu|(\Omega)-|\nabla^mv|(\Omega)\Bigr|+\Bigl|f(\nabla^mu)(\Omega)-f(\nabla^mv)(\Omega)\Bigr|
\end{align*}
where $f(x)=\sqrt{1+|x|^2}$ for $x\in\R^M$.
Then convergence with respect to this distance refines strict $BV$-convergence (see \cite{AFP}, Definition 3.14) and $C^\infty(\Omega)$  is a dense subspace of $\big(BV^m(\Omega),d_f(.,.)\big)$ (see \cite{DT}, Theorem 2.2).\\
To simplify matters, all of our results are formulated in terms of real valued functions and extend component-wise to the vector-valued case.

\subsection{Some auxiliary results on Sobolev functions}
\label{sec:3}
\begin{Prop}\label{Pr1}
Let $\Omega\subset\R^n$ be open and bounded with Lipschitz boundary and $u\in W^{m,p}(\Omega)$. With $T:W^{1,p}(\Omega)\rightarrow L^p(\partial \Omega,\mathcal{H}^{n-1})$ denoting the boundary operator for real-valued Sobolev  functions, we have that for any $\eps>0$ given, there is a smooth function $\varphi\in C^\infty(\Omega)\cap W^{m,p}(\Omega)$ with
\[
\|u-\varphi\|_{m,p;\Omega}<\eps
\]
and such that $T\nabla^ku=T\nabla^k\varphi$ for $0\leq k\leq m-1$, where the action of $T$ on a tensor-valued function is component-wise.
\end{Prop}
\begin{proof}
Exhaust $\Omega$ with open sets as given by
\[
\Omega_j:=\{x\in\Omega\;:\;\mbox{dist}(x,\partial \Omega)>1/j\},\quad j=1,2,...
\]
and consider the open covering of $\Omega$  through
\[
 A_1:=\Omega_2,\quad A_j:=\Omega_{j+1}-\overline{\Omega_{j-1}},\quad j=2,3,...
\]
Let $(\eta_j)_{j=1}^{\infty}$ be a partition of unity with respect to the covering $(A_j)_{j=1}^{\infty}$ and take a sequence $(\eps_j)_{j=1}^{\infty}$ of positive reals s.t. $(\mbox{spt }\eta_j)^{\eps_j}\Subset A_j$ and 
\[
\|\eta_ju-\rho_{\eps_j}\ast (\eta_ju)\|_{m,p;\Omega}<\eps/2^j.
\]
It is obvious  that $\varphi:=\sum_{j=1}^\infty \rho_{\eps_j}\ast (\eta_ju)$ is a smooth function which approximates $u$ in the right manner.\\
Now let $T_j:W^{1,p}(\Omega-\overline{\Omega_j})\rightarrow L^p(\partial \Omega)$ denote the trace operator on $W^{1,p}(\Omega-\overline{\Omega_j})$. Note that $T_ju_{|\Omega-\overline{\Omega_j}}=Tu$ whenever $u\in W^{1,p}(\Omega)$. Furthermore, since the trace operators are continuous, there are positive constants $c_j$ s.t.
\begin{equation}
 \|T_ju\|_{p;\partial\Omega}\leq c_j\|u\|_{1,p;\Omega-\overline{\Omega_j}}. \tag{1}
\end{equation}
With
\[
 a_j:=\frac{1}{\max\{c_i\;:\;i\leq j\}},
\]
we can choose $\eps_j$ small enough such that
\[
\|\eta_ju-\rho_{\eps_j}\ast (\eta_ju)\|_{m,p;\Omega}<a_j/2^j.
\]
Now let $0\leq k\leq m-1$. Thus $\nabla^ku\in W^{1,p}(\Omega,S^{k}(\R))$ and by (1) we have
\[
\|T\nabla^ku-T\nabla^k\varphi\|_{p;\partial\Omega}=\|T_j\nabla^ku_{|\Omega-\overline{\Omega_j}}-T_j\nabla^k\varphi_{|\Omega-\overline{\Omega_j}}\|_{p;\partial\Omega}
\]
\[
\leq c_j\|\nabla^ku-\nabla^k\varphi\|_{1,p;\Omega-\overline{\Omega_j}}
\]
\[
\leq c_j\sum_{l\geq j}\|\eta_lu-\rho_{\eps_l}\ast (\eta_lu)\|_{m,p;\Omega}<c_j\sum_{l\geq j}\frac{a_l}{2 ^l}\leq \frac{1}{2^{j-1}}.
\]

Since this holds for any $j\in\N$, the result follows.
\end{proof}

\begin{Prop}\label{Prop2}
 Let $\Omega\subset\R^n$ be open and $u\in W^{m,p}(\Omega)\cap L^q(\Omega)$. Then, for any $\eps>0$ given, there is a smooth function $\varphi\in C^\infty(\Omega)\cap W^{m,p}(\Omega)$ satisfying
\[
 \|u-\varphi\|_{m,p;\Omega}+\|u-\varphi\|_{q;\Omega}<\eps.
\]

\end{Prop}
\begin{proof}
 If we construct $\varphi$ in the same manner as in the prove of Proposition \ref{Pr1}, it follows trivially from the properties of mollification (see e.g. \cite{Ad}, Theorem 2.29)  that $\varphi$ approximates $u$ in $L^q(\Omega)$.
\end{proof}

\begin{Prop}\label{Lp}
Let $u\in L^p(\R^n)$. For $\alpha>1$ define $u_\alpha(x):=u(\alpha x)$. Then $u_\alpha\rightarrow u$ in $L^p(\Omega)$ for any sequence $\alpha\downarrow 1$ and any measurable set $\Omega\subset\R^n$.
\end{Prop}
\begin{proof} W.l.o.g. we assume $\Omega=\R^n$.
The set of smooth functions with compact support $C^\infty_0(\R^n)$ is dense in $L^p(\R^n)$. Thus, we can choose a sequence  $(\varphi_k)_{k\in\N}\subset C^\infty_0(\R^n)$ converging to $u$. Then $\varphi_k(\alpha x)$ approximates $u_\alpha$ in $L^p(\R^n)$ and the result follows since $\varphi_k(\alpha x)\rightarrow\varphi_k(x)$ converges uniformly for $\alpha\downarrow 1$ and  $k$ fixed.
\end{proof}

The following extension result will be a key tool towards proving approximation theorems in both $W^{m,p}(\Omega)\cap L^q(\Omega-D)$ and $BV^m(\Omega)\cap L^q(\Omega-D)$:
\begin{Prop} \label{Ext}
Let $\Omega\subset\mathbb{R}^n$ be open and bounded with minimally smooth boundary and $u\in W^{m,p}(\Omega)\cap L^q(\Omega)$. Then there is a continuous linear operator
$\mathfrak{E}$, mapping $u$ to a function $\tilde{u}\in W^{m,p}(\mathbb{R}^n)\cap L^q(\mathbb{R}^n)$ and such that $\tilde{u}=u$ (a.e.) on $\Omega$.
\end{Prop}
\begin{proof}
We claim, that the operator $\mathfrak{E}:W^{m,p}(\Omega)\rightarrow W^{m,p}(\R^n)$, as defined in part 3.3, pp. 189-192 of \cite{St} performs an extension in the right manner. Indeed, this is a mere consequence of the universality of this operator in the sense that it simultaneously extends all orders of differentiability by the same construction.
\end{proof}

\section{Proof of Theorem \ref{Res1}}
\label{sec:4}

We start by proving another version of Theorem \ref{Res1} under stronger assumptions on the geometry of $\Omega$ and $D$ in order to clarify the main idea and then  apply similar arguments to a more general setting. 

\begin{Lem}\label{Wkp}
Let $\Omega\subset\mathbb{R}^n$ be open and bounded  with minimally smooth boundary, $D\Subset\Omega$ an  open and precompact subset with Lipschitz boundary which is star-shaped with respect to a point $x_0\in D$ and $u\in W^{m,p}(\Omega)\cap L^q(\Omega-D)$. Given an arbitrary $\eps>0$, there is a function $\varphi\in C^\infty(\overline{\Omega})$ s.t.
\[
 \parallel u-\varphi \parallel_{m,p;\Omega}+\parallel u-\varphi \parallel_{q;\Omega-D}<\eps.
\]  
\end{Lem}
\begin{proof} W.l.o.g. we may assume $x_0=0$.\\
Applying Proposition \ref{Ext}, we can extend $u$ outside of $\Omega$ to a function $u'\in W^{m,p}(\R^n)\cap L^q(\R^n-D)$. Then, by Proposition \ref{Lp}, $u'_\alpha(x):=u'(\alpha x)$ converges to $u'$ in $W^{m,p}(\Omega)\cap L^q(\Omega-D)$ for $\alpha\downarrow 1$.
Fix $\alpha>1$ with 
\begin{equation}
\parallel u'-u'_\alpha \parallel_{m,p;\Omega}+\parallel u'-u'_\alpha\parallel_{q;\Omega-D}<\eps/3 \tag{1}
\end{equation}
Due to its star shape, $D_\alpha:=1/\alpha \,D$ is a precompact subset of $D$ and $u'_\alpha$ is $q$-integrable on $\R^n-D_\alpha$. By Proposition \ref{Prop2}, we can construct a smooth function $\varphi'\in C^\infty(\R^n-\overline{D_\alpha})$ with
\begin{equation}
\|u'_\alpha-\varphi'\|_{m,p;\R^n-\overline{D_\alpha}}+\|u'_\alpha-\varphi'\|_{q;\R^n-D_\alpha}<\eps/3 \tag{2}
\end{equation}
and such that $T\nabla^k\varphi'=T\nabla^ku'_\alpha$ in $L^p(\partial D_\alpha,\mathcal{H}^{n-1})$ for every $0\leq k\leq m-1$. Consequently, $\varphi'$ can be extended to $D_\alpha$ by ${u'_\alpha}_{|D_\alpha}$ to a function $v\in W^{m,p}(\Omega)\cap L^q(\Omega-D_\alpha)$. On $\overline{D}$, we can construct a smooth function $\varphi''\in C^\infty(\overline{D})$ with
\begin{equation}
 \|v-\varphi''\|_{m,p;D}<\eps/3 \tag{3}
\end{equation} 
and such that $\partial^\nu\varphi''_{|\partial D}=\partial^\nu\varphi'_{|\partial D}$ for every multi-index $\nu\in\N_0^n$. Therefore, and by (1)-(3)
\begin{equation*}
 \varphi(x):=
 \begin{cases}
  \varphi''(x),\quad &x\in D,\\
  \varphi'(x),\quad &x\in\overline{\Omega}-D
 \end{cases}
\end{equation*}
is a smooth function that approximates $u$ in the right manner.
\end{proof}

\noindent We now come to the \textit{proof of Theorem \ref{Res1}:}

Let $\{x_1,\,x_2,\,x_3,...\}\subset\partial D$ be a dense subset of $\partial D$. For every $i\in\N$ choose an  open ball $B_{r_i}(x_i)$ such that $B_{r_i}(x_i)\cap D$ is Lipschitz-equivalent to $B_1(0)\cap\R^{n-1}\times (-\infty,0]$ via a bi-Lipschitz-map $\phi_i:B_{r_i}(x_i)\rightarrow B_1(0)$ and such that $\inf_ir_i>0$. Let $p_i$ denote the preimage of $(0,...,0,-1)$ with respect to $\phi_i$. W.l.o.g. we can assume $p_i=0$ for $i$ fixed. Note that $B_{r_i}(x_i)$ is star shaped with respect to $p_i$.\\
Now let $\eta_i\in C^\infty_0(B_{r_i}(x_i))$ be a smooth function with $0\leq \eta_i\leq 1$, $\eta_i\equiv 1$ on $B_{r_i/2}(x_i)$. We successively construct a sequence $(u_i)_{i=1}^\infty$ of $W^{m,p}(\Omega)\cap L^q(\Omega-D)$-functions in the following way:\\
For $i=1$, take $\alpha_1>1$ small enough such that $u_1(x):=(\eta_1 u)(\alpha_1 x)+(1-\eta_1(x))u(x)$ fulfills
\[
 \|u-u_1\|_{m,p;\Omega}+\|u-u_1\|_{q;\Omega-D}<\eps/2.
\]
Then (provided $\alpha_1$ is small enough) $u_1$ is $q$-integrable outside a proper subset of $D$, with positive distance from $\partial D$ near $\partial D\cap B_{r_1/4}(x_1)$. In the second step, we find $\alpha_2>1$ for which the function $u_2:=(\eta_2 u_1)(\alpha_2 x)+(1-\eta_2(x))u_1(x)$ satisfies
\[
 \| u_1-u_2\|_{m,p;\Omega}+\| u_1-u_2\|_{q;\Omega-D}<\eps/4.
\]
Then $u_2$ is $q$-integrable outside a proper subset of $D$, with positive distance from $\partial D$ near  $\partial D\cap \big(B_{r_1/4}(x_1)\cup B_{r_2/4}(x_2)\big)$.\\
By continuing this process, we recursively define a sequence $(u_i)$ s.t.
\[
 \|u_{i-1}-u_i\|_{m,p;\Omega}+\| u_{i-1}-u_i\|_{q;\Omega-D}<\eps/2^i
\]
and $u_i$ is  $q$-integrable beyond $\partial D\cap \big(\bigcup_{j=1}^i B_{r_j/4}(x_j)\big)$, i.e. the domain of $q$-integrability is enlarged gradually to the inside of $D$. Since $\partial D$ is compact in $\Omega$, after finitely many steps $N$,  $\bigcup_{i=1} ^NB_{r_{i}/4}$ covers $\partial D$. Then $u_{N}$ is a function with
\[
\|u-u_{N}\|_{m,p;\Omega}+\|u-u_{N}\|_{q;\Omega}<\eps
\]
and that is $q$-integrable outside an inner parallel set of $D$. From this point on, the result  follows by the same arguments as used in the proof of Lemma \ref{Wkp}.
\qed

\section{Proof of Theorem \ref{Res2}}
\label{sec:5}

In this section we are concerned with generalizing our previous results for Sobolev functions towards the spaces $BV^m(\Omega)\cap L^q(\Omega-D)$.

\begin{Def}
In the following, we will keep saying "$\varphi$\textit{ approximates } $u\in BV^m(\Omega)\cap L^q(\Omega-D)$ \textit{ in the sense of }$(\mathcal{A}_\eps)$"  for a given $\eps>0$, if $\varphi$ approximates $u$ with respect to the metric $d_f(.,.)$ as well as in $L^q(\Omega-D)$:
\begin{equation*} 
(\mathcal{A}_\eps)
 \begin{cases}
\|u-\varphi\|_{m-1,1;\Omega}+\|u-\varphi\|_{q,\Omega-D}\\
\,\\
+\bigl||\nabla^mu|(\Omega)-|\nabla^m\varphi|(\Omega)\bigr|\\
\,\\
+\Bigl|\sqrt{1+|\nabla^mu|^2}(\Omega)-\intop_{\Omega}\sqrt{1+|\nabla^m\varphi|^2}dx\Bigr|<\eps.
 \end{cases}
\end{equation*}
\end{Def}

Notice, that corresponding versions of Proposition \ref{Pr1} and  \ref{Prop2} can be proven in the context of $BV^m(\Omega)$:

\begin{Prop}\label{BVmq}
 Let $\Omega\subset\R^n$  have $C^1$-boundary\footnote{The author is not particularly sure to what extent it is necessary to request actual smoothness of the boundary, since Demengel and Temam in \cite{DT} only speak of a 'sufficiently smooth' boundary, but it seems to be adequate to assume it to be once differentiable.} and $u\in BV^m(\Omega)\cap L^q(\Omega)$. Then, for any $\eps>0$ given there is a smooth function $\varphi\in C^\infty(\Omega)\cap BV^m(\Omega)$ satisfying
\[
 d_f(u,\varphi)+\|u-\varphi\|_{q;\Omega}<\eps.
\]
\end{Prop}
\begin{proof}
 In \cite{DT}, Theorem 2.2 it is shown, that $C^\infty(\Omega)$ lies dense in $BV^m(\Omega)$ with respect to the distance $d_f(.,.)$. The construction of such a smooth approximation follows basically the same steps as in case of a Sobolev function (i.e. the classical Meyers-Serrin argument (see \cite{MS}) as also seen in the proof of Proposition \ref{Pr1}), and thus it is clear that additional integrability constraints are respected by the approximation thanks to the properties of mollification.
\end{proof}
\begin{Prop}\label{Pr1BV}
Let $\Omega\subset\R^n$ be open and bounded with Lipschitz boundary and $u\in BV^m(\Omega)$. Define $T:W^{1,1}(\Omega)\rightarrow L^1(\partial\Omega,\mathcal{H}^{n-1})$ as in Proposition \ref{Pr1} and let $S:BV(\Omega)\rightarrow L^1(\partial\Omega,\mathcal{H}^{n-1})$ denote the trace operator on $BV(\Omega)$. Then, for any $\eps >0$ there is a smooth function $\varphi\in C^\infty(\Omega)\cap BV^m(\Omega)$ which approximates $u$ in the sense of $(\mathcal{A}_\eps)$ and such that
\[
 T\nabla^ku=T\nabla^k\varphi,\,\text{ for all }0\leq k<m-2\quad\text{and}\quad S\nabla^{m-1}u=S\nabla^{m-1}\varphi
\]
in $L^1(\partial\Omega,\mathcal{H}^{n-1})$ (hold in mind that $T$ and $S$ act component-wise on tensor-valued functions).
\end{Prop}

\begin{proof}
The result follows by the same arguments we used in the proof of Proposition \ref{Pr1} since by Theorem 3, page 483 in \cite{GMS}, $S$ is continuous with respect to the metric $d_f(.,.)$ (see also \cite{Giu}, Theorem 2.11 and Remark 2.12 as well as \cite{DT}, Theorem 2.3).
\end{proof}

\begin{Cor}
 Let $\Omega\subset\R^n$ be open and bounded with $C^1$-boundary and $u\in BV^m(\Omega)\cap L^q(\Omega)$. Then there is a function $\tilde{u}\in BV^m(\R^n)\cap L^q(\R^n)$ such that $u=\tilde{u}$ a.e. on $\Omega$ and 
\[
 |\nabla^m\tilde{u}|(\partial\Omega)=0.
\]
\end{Cor}
\begin{proof}
 According to Propositions \ref{Pr1BV} and \ref{BVmq} above, we can choose a function $\varphi\in C^\infty(\Omega)\cap BV^m(\Omega)\cap L^q(\Omega)$ with $T\nabla^{k}\varphi=T\nabla^ku$ for $0\leq k\leq m-2$ and $S\nabla^{m-1}\varphi=S\nabla^{m-1}u$ in $L^1(\partial \Omega,\mathcal{H}^{n-1})$. In particular, $\varphi\in W^{m,1}(\Omega)\cap L^q(\Omega)$ and we can  therefore apply Proposition \ref{Ext} to extend $\varphi$ to a function $\tilde{\varphi}\in W^{m,1}(\R^n)\cap L^q(\R^n)$. But then
\begin{equation*}
\tilde{u}(x):=
 \begin{cases}
  u(x),\quad &x\in\Omega,\\
  \tilde{\varphi}(x),\quad &x\in\R^n-\Omega
 \end{cases}
\end{equation*}
is an extension of $u$ as claimed.
\end{proof}

\noindent With these results at hand, there now follows the \textit{proof of Theorem \ref{Res2}}:

 Without loss of generality, we may assume $x_0=0$.\\
 By Proposition \ref{BVmq} we can construct a smooth function $\psi\in C^\infty(\Omega-\overline{D})$ having the same traces as $u$ on $\partial D$ at any order and with
\begin{equation}
d_f(u_{|\Omega-D},\psi)+\|u-\psi\|_{q;\Omega-D}<\eps/3. \tag{1}
\end{equation}
In particular, $\psi$ is in $W^{m,1}(\Omega-\overline{D})\cap L^q(\Omega-D)$ and by Proposition \ref{Ext}, we can extend $\psi$ outside of $\Omega$ to a function $\psi'\in W^{m,1}(\R^n-\overline{D})\cap L^q(\R^n-D)$. 
Due to Proposition \ref{Pr1BV}, the function $\psi'$ can be extended by $u_{|D}$ to a function $u'$ in $BV^m(\R^n)\cap L^q(\R^n-D)$ s.t.
\begin{equation}
|\nabla^mu'|(\partial D)=|\nabla^mu|(\partial D) \tag{2}
\end{equation}
and since $|\nabla^mu|(\mathcal{N})=\sqrt{1+|\nabla^mu|^2}(\mathcal{N})$ for any $\mathcal{L}^n$-null set $\mathcal{N}$ we also get
\begin{equation}
\sqrt{1+|\nabla^mu|^2}(\partial D)=\sqrt{1+|\nabla^mu'|^2}(\partial D). \tag{3}
\end{equation}
Altogether, (1)-(3) imply that $u'$ approximates $u$ in the sense that
\[
d_f(u,u')+\|u-u'\|_{q;\Omega-D}<\eps/3.
\]
 Now we consider $u'_\alpha(x):=u'(\alpha x)$ for $\alpha>1$. Then, by the star shape of $D$, $u'_\alpha$ is $q$-integrable outside of $D_\alpha:=(1/\alpha)D\Subset D$. \\
It obliges to show $u'_\alpha\rightarrow u'$ in the sense of $(\mathcal{A}_\eps)$ for $\alpha\downarrow 1$.\\
With $h:\R^n\rightarrow\R^n$, $x\mapsto (1/\alpha)\,x$, we have $\nabla^m(u'_\alpha)=\alpha^{m-n}h_*\nabla^mu$, where $h_*\mu(B):=\mu(h^{-1}(B))$ denotes the image measure.\\ 
Further we get:
\begin{align*}
 |\nabla^m u'_\alpha|(\Omega)=\sup\Biggl\{\intop_{\Omega}u'(\alpha x)\bigg( \sum_{|\nu|=m}\partial^\nu g_\nu(x)\bigg) dx\;:\;g\in C^m_0(\Omega,\R^M), \|g\|_\infty\leq 1\Biggr\}\\
=\alpha^{-n}\sup\Biggl\{\intop_{\alpha\Omega}u'(x)\bigg( \sum_{|\nu|=m}\partial^\nu g_\nu\bigg)(x/\alpha) dx\;:\;g\in C^m_0(\Omega,\R^M), \|g\|_\infty\leq 1\Biggr\}\\
=\alpha^{m-n}\sup\Biggl\{\intop_{\alpha\Omega}u'(x)\bigg( \sum_{|\nu|=m}\partial^\nu g_\nu(x/\alpha)\bigg)dx\;:\;g\in C^m_0(\Omega,\R^M), \|g\|_\infty\leq 1\Biggr\}\\
 = \alpha^{m-n}|\nabla^m u'|(\alpha\Omega)\xrightarrow{\alpha\downarrow 1}|\nabla^m u'|(\overline{\Omega})=|\nabla^m u'|(\Omega),
\end{align*}
since $u'\in W^{m,1}(\R^n-\overline{D})$ and therefore $|\nabla^m u'|(\partial\Omega)=0$. This proves
\[
 \underset{\alpha\downarrow 1}{\limsup\,}|\nabla^m u'_\alpha|(\Omega)\leq|\nabla^m u'|(\Omega)
\]
and convergence follows from $\nabla^{m-1}u'_\alpha\xrightarrow{\alpha\downarrow 1}\nabla^{m-1}u'$ in $L^1(\R^n)$ and lower semi-continuity of the total variation.\\
Moreover, if 
\[
\nabla^mu'=\nabla^m_au'\mathcal{L}^n+\nabla^m_su' 
\]
denotes the Lebesgue-decomposition of the tensor valued Radon measure $\nabla^m u'$, we have that
\[
 \alpha^{m-n}h_*\nabla^mu'=\alpha^m\nabla^m_au'\circ h^{-1}\mathcal{L}^n+\alpha^{m-n}h_*\nabla^m_su'
\]
is the Lebesgue-decomposition of $\nabla^mu'_\alpha$, and by definition it follows:
\begin{align*}
 \sqrt{1+|\nabla^mu'_\alpha|^2}(\Omega)=\intop_{\Omega}\sqrt{1+|\alpha^m\nabla^m_au'(\alpha x)|^2}dx+\alpha^{m-n}|h_*\nabla^m_su'|(\Omega).
\end{align*}
As above, for the total variation of the singular part we have
\[
\alpha^{m-n}|h_*\nabla^m_su'|(\Omega)=\alpha^{m-n}|\nabla^m_su'|(\alpha\Omega)\xrightarrow{\alpha\downarrow 1}|\nabla^m_su'|(\overline{\Omega})=|\nabla^m_su'|(\Omega).
\]
To the first part, we can apply the transformation formula:
\[
 \intop_{\Omega}\sqrt{1+|\alpha^m\nabla^m_au'(\alpha x)|^2}dx=\alpha^{-n} \intop_{\alpha\Omega}\sqrt{1+|\alpha^m\nabla^m_au'(x)|^2}dx.
\]
Due to $\alpha^m\nabla^m_au'\xrightarrow{\alpha\downarrow 1} \nabla^m_au'$ pointwise a.e. and $|\alpha^m\nabla^m_au'(x)|\leq 2|\nabla^m_au'(x)|$ (we may assume  $\alpha^m<2$), by Lebesgue's theorem on dominated convergence we conclude
\[
\alpha^{-n} \intop_{\alpha\Omega}\sqrt{1+|\alpha^m\nabla^m_au'(x)|^2}dx\xrightarrow{\alpha\downarrow 1}  \intop_{\Omega}\sqrt{1+|\nabla^m_au'(x)|^2}dx.
\]
Hence, we can choose $\alpha>1$ small enough with  
\begin{equation}
d_f(u',u'_\alpha)+\|u'-u'_\alpha\|_{q,\Omega-D}<\eps/3 \tag{4}
\end{equation}
and $u'_\alpha$ is $q$-integrable outside $D_\alpha$. From that point on, we may proceed just like in the proof of Lemma \ref{Wkp} and construct a smooth function $\varphi\in C^\infty(\overline{\Omega})$ with
\begin{equation}
 \|u'_\alpha-\varphi\|_{q;\Omega-D}+d_f(u'_\alpha,\varphi)<\eps/3 \tag{5}.
\end{equation}
by conjoining $C^\infty$-approximations of $u'_\alpha$ on $\R^n-\overline{D_\alpha}$ and $D$.
Altogether, we have that $\varphi$ approximates $u$ as claimed.
\qed

\begin{Rem}
 One might expect that, using similar arguments as in the proof of Theorem \ref{Res1}, we can generalize the above result towards weaker assumptions on $\Omega$ and $D$; but this is not the case. This seems to ground on the fact that the metric $d_f(.,.)$ is not translation invariant, and addition does  not act  continuously w.r.t. the topology it induces on $BV^m(\Omega)$. Put simply: minor changes of a function $u\in BV^m(\Omega)$ on a small set can have a major effect on its global behavior.
\end{Rem}

\begin{section}{Boundary traces of $W^{m,p}(\Omega)\cap L^q(\Omega)$-functions, proof of Theorem \ref{Res3}}
 Revising the steps in the proof of our approximation result in the Sobolev context, we find that our method largely relies on the extension result \ref{Ext}, being the reason for that we have to presume  $D$ to be compactly contained in $\Omega$ which guarantees $q$-integrability near the boundary $\partial\Omega$. Hence we could prove Theorem \ref{Res1} in a much broader setting if already any $W^{m,p}(\Omega)$-function could be extended from $\Omega$ to $\R^n$ by a $W^{m,p}(\R^n)\cap L^q(\R^n)$-function. This leads to the question, whether the images of $W^{m,p}(\Omega)$ and $W^{m,p}(\Omega)\cap L^q(\Omega)$ under the Sobolev trace map $T:W^{1,p}(\Omega)\rightarrow L^p(\partial\Omega)$ coincide for $q>\frac{np}{n-mp}$. This section is devoted to the prove of Theorem \ref{Res3}, which gives a negative answer if $p>1$ or $m>1$.

In what follows, let $\Omega$ be the 'upper' half-space $\R^{n-1}\times (0,\infty)$ in $\R^n$ and $T:W^{1,p}(\Omega)\rightarrow L^p(\partial\Omega)=L^p(\R^{n-1})$ denote the trace map for Sobolev functions. A classical result by Gagliardo in \cite{Ga} is, that only for $p=1$ this map is onto. For $p>1$, the investigation of the image of $T$ in $L^p(\R^{n-1})$ led to the idea of fractional Sobolev spaces (often referred to as Sobolev-Slobodeckij spaces) $W^{s,p}(\Omega)$ for arbitrary non-integer $s>0$. With these at hand, the exact trace of $W^{1,p}(\Omega)$ is given by $W^{1-\nicefrac{1}{p},p}(\R^{n-1})$. 

\noindent\textit{Proof of Theorem \ref{Res3}:}

\noindent\textit{ad (i):}\quad Let $f\in L^1(\R^{n-1})$ be an arbitrary function on the boundary of $\Omega$. For a given $q\geq 1$, we are going to construct a function $u\in W^{1,1}(\Omega)\cap L^q(\Omega)$ with $T(u)=f$:

Let $(\varphi_k)_{k=1}^\infty\subset C^\infty_0(\R^{n-1})$ be a sequence of smooth functions with compact support, which approximates $f$ in the following way:
\begin{align}
&\|f-\varphi_k\|_{1;\R^{n-1}}\xrightarrow{k\rightarrow\infty}0 \tag{1}\\
&\|\varphi_{k+1}-\varphi_{k}\|_{1,\R^{n-1}}<2^{-k}. \tag{2}
\end{align}
Since $C^\infty_0(\R^{n-1})\subset L^q(\R^{n-1})$, we can further choose a monotonously decreasing null-sequence $(\delta_k)_{k=1}^\infty$ such that
\begin{align}
 &\delta_k\|\varphi_k\|_{q;\R^{n-1}}<2^{-k},\tag{3}\\
 &\delta_k\|\nabla\varphi_k\|_{1;\R^{n-1}}<2^{-k}\quad\text{ for }k=1,2,...\tag{4},\\
 &\eps_0:=\sum_{k=1}^\infty \delta_k<\infty \tag{5}.
\end{align}

Setting $\eps_k:=\sum_{i=k+1}^\infty \delta_i$ for $k=1,2,...$, we now define a function $\tilde{u}(x,t)$ at a point $(x,t)\in \R^{n-1}\times (0,\infty)$ piecewise by
\begin{align*}
 \tilde{u}(x,t):=\begin{cases}
          0,\quad&\text{ for }\eps_0\leq t,\\
          \varphi_1(x),\quad&\text{ for }\eps_1\leq t<\eps_0,\\
           \hspace{0.5cm}\vdots\\
           \varphi_k(x),\quad&\text{ for }\eps_{k}\leq t<\eps_{k-1},\\
           \hspace{0.5cm}\vdots
     \end{cases}
\end{align*}

It is readily seen from (3), that $\tilde{u}\in L^q(\Omega)$. Furthermore we claim  $\tilde{u}\in BV(\Omega)$: Since $\tilde{u}$ has only jump-type discontinuities concentrated on the set 
\[
S_{\tilde{u}}=\bigcup_{i=0}^\infty \left(\R^{n-1}\times \{\eps_i\}\right),
\]
which is countably $(n-1)$-rectifiable and $\tilde{u}$ is differentiable outside $S_{\tilde{u}}$ with
\begin{align*}
 \nabla\tilde{u}(x,t):=\begin{cases}
          0,\quad&\text{ for }\eps_0< t,\\
          \nabla\varphi_1(x)\oplus 0,\quad&\text{ for }\eps_1\leq t<\eps_0,\\
           \hspace{1cm}\vdots\\
           \nabla\varphi_k(x)\oplus0,\quad&\text{ for }\eps_{k}\leq t<\eps_{k-1},\\
           \hspace{1cm}\vdots
     \end{cases}
\end{align*}
the total variation of $\tilde{u}$ can be calculated to
\[
|\nabla\tilde{u}|(\Omega)=\sum_{k=1}^\infty\delta_k\|\nabla \varphi_k\|_{1,\R^{n-1}}+\|\varphi_1\|_{1;\R^{n-1}}+\sum_{k=1}^\infty\|\varphi_{k+1}-\varphi_k\|_{1;\R^{n-1}}
\]
which is finite by (2)+(4).\\
Thus, by the properties of mollification
\[
 u(x,t):=\big(\rho_{\nicefrac{t}{2}}\ast \tilde{u}\big)(x,t)
\]
defines a $W^{1,1}(\Omega)\cap L^q(\Omega)$-function which has boundary trace $f$ on $\R^{n-1}$ by construction.

\noindent\textit{ad (ii):}\quad In order to prove the non-surjectivity in the case $p>1$, we make use of the following generalization  of the classical Gagliardo-Nirenberg inequality towards fractional Sobolev spaces (see \cite{BM}, Corollary 2):

\begin{Lem}
Let $1<p,q<\infty$, $0< s< 1$ and $u\in W^{1,p}(\Omega)\cap L^q(\Omega)$. Then there is a constant $C=C(p,q,s)>0$ such that 
\[
\|u\|_{s,p(s);\Omega}\leq C \|u\|_{q;\Omega}^{1-s}\|u\|_{1,p;\Omega}^s
\]
where
\begin{equation*}
\frac{1}{p(s)}=\frac{s}{p}+\frac{1-s}{q}.\label{gni}
\end{equation*}
\end{Lem}
We are going to show, that whenever $q>\frac{np}{n-p}$, we can choose $0<s_0<1$ s.t. $p(s_0)>\frac{(n-1)p}{n-p}$ and $s_0>\frac{1}{p(s_0)}$. Then by the above Lemma, every function $u\in W^{1,p}(\Omega)\cap L^q(\Omega)$ is an element of $W^{s_0,p(s_0)}(\Omega)$ and consequently, by \cite{Tr}, Theorem 2.7.2 it has a boundary trace in $L^{p(s_0)}(\R^{n-1})$. Notice, that due to the fact that the trace operator is in any case defined through the continuation of the trivial map $u\mapsto u_{|\partial\Omega}$ on the dense subspace $C^\infty(\overline{\Omega})\cap W^{1,p}(\Omega)$ (and $C^\infty(\overline{\Omega})\cap W^{s_0,p(s_0)}(\Omega)$, respectively), in $L^p(\partial\Omega)$ the trace of $u$ as a $W^{1,p}(\Omega)\cap L^q(\Omega)$-function will be the same as the trace of $u$ as an $W^{s_0,p(s_0)}(\Omega)$-function, since by the above inequality every sequence of smooth functions approximating $u$ in $W^{1,p}(\Omega)\cap L^q(\Omega)$ approximates $u$ as an element of $W^{s_0,p(s_0)}(\Omega)$ as well.  But since $p(s_0)$ exceeds the maximal exponent from  Sobolev's embedding theorem for traces (see \cite{Ad}, Theorem 5.4 Case A) which is proven to be optimal via a counterexample in \cite{Ad}, example 5.25, we conclude that there are indeed traces in $T(W^{1,p}(\Omega))$ which do not come from a $W^{1,p}(\Omega)\cap L^q(\Omega)$-function.

Solving 
\[
 \frac{1}{p(s)}=\frac{s}{p}+\frac{1-s}{q}<s
\]
for $s$ yields
\[
s>\frac{p}{pq+p-q}.
\]
Surely, $p(s)$ becomes maximal for $s$ minimal, so we put
\[
s_{min}:=\frac{p}{pq-q+p}<1
\]
and calculate $p_{max}:=p(s_{min})$ to
\[
p_{max}=q\left(1-\frac{1}{p}\right)+1.
\]
Since $q>\frac{np}{n-p}$ by assumption, we have
\[
p_{max}>\frac{(n-1)p}{n-p}.
\]
Hence we can choose $s_{min}<s_0<1$ small enough s.t. $\frac{(n-1)p}{n-p}<p(s_0)<p_{max}$.

\begin{Rem}
From part $(ii)$ of the theorem we conclude, that for $m>1$ the image  $T(W^{m,p}(\Omega)\cap L^q(\Omega))$ in $L^p(\partial\Omega)$ is always a proper subspace of $T(W^{m,p}(\Omega))$, since $W^{m,p}(\Omega)$ is embedded into some $W^{1,p'}(\Omega)$ for $p'>1$ via  Sobolev's embedding theorem. 
\end{Rem}

\end{section}



\end{document}